\documentclass[a4paper,11pt]{amsart}
\usepackage{amsmath,amsfonts,amscd,amssymb,latexsym,amsthm,diagrams}
\usepackage[all]{xy}

\newtheorem{lemma}[equation]{Lemma}
\newtheorem{df}[equation]{Def\/inition}
\newtheorem{prop}[equation]{Proposition}
\newtheorem{cor}[equation]{Corollary}

\newcommand{\eqn}{\begin{equation}}
\newcommand{\eeqn}{\end{equation}}

\DeclareMathOperator{\ord}{ord}

\DeclareMathOperator{\Spec}{Spec}

\DeclareMathOperator{\NP}{NP}  

\begin{document}


\title{Overconvergent Witt Vectors}
\date{}
\author{Christopher Davis}
\address{Massachusetts Institute of Technology, Dept of Mathematics, Cambridge, MA 02139}
\email{davis@math.mit.edu}

\author{Andreas Langer}
\address{University of Exeter, Mathematics, Exeter EX4 4QF, Devon, UK}
\email{a.langer@exeter.ac.uk}

\author{Thomas Zink}
\address{Fakult\"at f\"ur Mathematik, Universit\"at Bielefeld, Postfach 100131, D-33501 Bielefeld}
\email{zink@math.uni-bielefeld.de}

\begin{abstract}
Let $A$  be a finitely generated algebra over a field $K$ of
characteristic $p >0$. We introduce a subring $W^{\dagger}(A) \subset
W(A)$, which we call the ring of overconvergent Witt vectors and prove
its basic properties.  In a subsequent paper we use the results to
define an overconvergent de Rham-Witt complex for smooth varieties
over $K$ whose hypercohomology is the rigid cohomology.
\end{abstract}

\maketitle

\numberwithin{equation}{section}
\numberwithin{thmsub}{equation}

\section*{Introduction}
 

Overconvergent Witt vectors were used by de Jong in his proof of Tate's
conjecture on homomorphisms of $p$-divisible groups \cite{J} and in
Kedlaya's work on the Crew conjecture. In \cite{DLZ} we define a de
Rham-Witt complex over the ring of overconvergent Witt vectors which
computes the rigid cohomology of smooth varieties of a perfect field
of characterisitc $p>0$. For this it is necessary to consider
overconvergent Witt vectors in a  more general setting.

Let $A$ be a finitely generated algebra of a field $K$ of
characteristic $p$. Let $W(A)$ be the ring of Witt vectors with
respect to $p$. 
We define a subring $W^{\dagger}(A)\subset W(A)$
which we call the ring 
of overconvergent Wittvectors. Let $A=K[T_1,\dots, T_d]$ be the
polynomial ring. We say that a Witt vector $(f_0,f_1,f_2, \ldots ) \in
W(A)$ is overconvergent, if there is a real number $\varepsilon >0$  and a
real number $C$ such that 
\begin{displaymath}
   m - \varepsilon p^{-m} \deg f_m \geq C, \quad \text{for all} \; m
   \geq 0.
\end{displaymath} 
The overconvergent Witt vectors form a subring $W^{\dagger}(A) \subset
W(A)$. 

There is a natural morphism
from the ring of restricted power series
    $$W(K)\{T_1,\dots, T_d\}\to W(A),$$
which maps $T_i$ to its Teichm\"uller representative $[T_i]$.

The inverse image of $W^{\dagger}(A)$ is the set of those power series which
converge in some neighborhood of the unit ball.
This is the weak completion $A^{\dagger}$ of $W(K)[T_1,\dots,T_d]$ in the
sense of Monsky and Washnitzer. We note that the bounded Witt vectors
used by Lubkin \cite{L} are different from the overconvergent Witt
vectors.  

If $A \rightarrow B$ is a surjection
of finitely generated $K$-algebras, we obtain by definition a
surjection of the rings of overconvergent Witt vectors 
\begin{displaymath}
   W^{\dagger}(A) \rightarrow W^{\dagger}(B).
\end{displaymath} 

We prove here basic properties of overconvergent Witt vectors which
we use in \cite{DLZ}:

Let $A \subset B$ be two smooth $K$-algebras. Then

\begin{displaymath}
   W^{\dagger}(A) = W(A) \cap W^{\dagger}(B)
\end{displaymath}
(see: Proposition \ref{LOK18p}).

\smallskip

Further we show (see Corollary \ref{endl-et}):
Let $A$ be a finitely generated algebra over  $K$. Let $B =
A[T]/(f(T))$ be a finite \'etale $A$-algebra, where $f(T) \in A[T]$ is
a monic polynomial of degree $n$, such that $f'(T)$ is a unit in $B$. 

We denote by $t$ the residue class of $T$ in $B$. 
Then $W^{\dagger} (B)$ is finite and \'etale over $W^{\dagger} (A)$,
and the elements $1, [t], [t]^2 \ldots, [t]^{n-1}$ form a basis of the
$W^{\dagger}(A)$-module $W^{\dagger}(B)$.

\smallskip
 
Finally we prove that $W^{\dagger}(A) \rightarrow A$ satisfies
Hensel's lemma (see Proposition \ref{Hensel}).


\section{Pseudovaluations}

We set $\bar{\mathbb{R}} = \mathbb{R} \cup \{\infty\} \cup
\{-\infty\}$ with its natural order.
\begin{df}
Let $A$ be an abelian group. An order function is a function
\begin{displaymath}
\nu : A \rightarrow \bar{\mathbb{R}},
\end{displaymath}
such that $\nu (0) = \infty$ and such that for arbitrary $a, b \in A$: 
\begin{displaymath}
\nu (a \pm b) \geq \min \{ \nu (a),\nu (b)\}.
\end{displaymath}
\end{df}
An order function is the same thing as a decending filtration 
of $A$ by subgroups $F^rA$ indexed by $r \in \bar{\mathbb{R}}$, with
the property $\cap_{s <r} F^sA = F^rA$. 
 
In particular the inequality above is an equality if $\nu (a) \neq \nu
(b)$. Moreover we have $\nu (a) = \nu (-a)$. 

Let $\phi: A \rightarrow B$ be a surjective homomorphism of abelian
groups. Then we 
define the quotient $\bar{\nu} : B \rightarrow \bar{\mathbb{R}}$ by:
\begin{equation}\label{A19e}
\bar{\nu} (b) = \sup \{\nu(a) \;|\; a \in A, \; \phi (a) = b\}.
\end{equation}
This is again an order function.

We define an order function $\nu^n$ on the direct sum $A^n$ as follows:

\begin{equation}\label{A11e}
   \nu^n ((a_1, \ldots, a_n)) = \min_{i} \{ \nu(a_i) \}.
\end{equation}

\begin{df}\label{A11d}
Let $A$ be a ring with $1$. A pseudovaluation $\nu$ on $A$ is an order
function on the additive group
\begin{displaymath}
\nu : A \rightarrow \bar{\mathbb{R}} 
\end{displaymath} 
such that the following properties hold
\begin{enumerate}
\item[1)] $\nu(1) = 0$,
\item[2)] $\nu (ab) \geq \nu(a) + \nu(b)$, if $\nu (a) \neq -\infty$
  and $ \nu (b) \neq -\infty$. 
\end{enumerate}
\end{df}
We call $\nu$ proper if it doesn't take the value $-\infty$. 
We call $\nu$ negative if $\nu$ is proper and $\nu (a) \leq 0$ for all
$a \in A, a \neq 0$. If $\nu$ is proper and $2)$ is an equality, $\nu$
is called a valuation. On each ring $A$ we have the trivial
valuation: $\nu (a) = 0$ for $a \neq 0$.

If $\phi : A \rightarrow B $ is a surjective ring homomorphism.
Let $\nu$ be a pseudovaluation on $A$. Let $\bar{\nu}$ the induced
order function on $B$. If $\bar{\nu}(1) \neq \infty$ then $\bar{\nu}$
is a pseudovaluation. In particular this is the case if $\nu$ is
negative. 

{\bf Example 1:} Let $R$ be a ring with a negative pseudovaluation
$\mu$. Consider the polynomial ring $A = R[T_1, \ldots T_m]$. 
Let $d_1 >0, \ldots, d_m >0$ be real numbers. Then we define a
valuation 
on $A$ as follows: For a polynomial

\begin{displaymath}
f = \sum_{k} c_k T_1^{k_1}\cdot \ldots \cdot T_m^{k_m}   
\end{displaymath}
we set 

\begin{equation}\label{A9e}
   \nu(f) = \inf \{\mu(c_k) - k_1 d_1 - \ldots - k_m d_m \}.
\end{equation}
This is a valuation if $\mu$ is a valuation.
We often consider the case where $R$ is an integral domain and $\mu$
is the trivial valuation. If moreover $d_i = 1$ we call $\nu$ the
standard degree valuation.

We are interested in pseudovaluations up to equivalence:
\begin{df}
  Let $\nu_1, \nu_2 : A \rightarrow \mathbb{R} \cup \infty$ be two
  functions such that $\nu_{i} \neq 0$ for all $a \in A$. We say that
  they are linearly equivalent, if there are real numbers $c_1 > 0,
  c_2 > 0, d_1 \geq 0, d_2 \geq 0$ such that for all $a \in A$:
\begin{displaymath}
\begin{array}{ccc}
\nu_1 (a) & \geq & c_2\nu_2 (a) - d_2\\
\nu_2 (a) & \geq & c_1\nu_1 (a) - d_1
\end{array}
\end{displaymath}
\end{df}

In Example 1 (\ref{A9e}) we obtain for different choices of the
numbers $d_i$ linearly equivalent negative pseudovaluations.  The
equivalence class of $\nu$ doesn't change if we replace $\mu$ by an
equivalent negative pseudovaluation.

Let $\nu_1$ and $\nu_2$ be two negative pseudovaluations on $A$.  If
$A \rightarrow B$ is a surjective ring homomorphism, then the
quotients $\bar{\nu}_1$ and $\bar{\nu_2}$ are again linearly
equivalent.

\begin{prop}\label{A1p}
Let $\mu$ be a negative pseudovaluation on a ring $R$. We consider a
surjective 
ring homomorphism $\phi: R[T_1, \ldots, T_m] \rightarrow R[S_1,
\ldots, S_n]$. Let $\nu_T$ be a pseudovaluation on
$R[T_1, \ldots, T_m]$ and let $\nu_S$ be a pseudovaluation on 
$R[S_1, \ldots, S_n]$ as defined by (\ref{A9e}). Then the quotient of
$\nu_T$ with respect to $\phi$ is a pseudovaluation which is linearly
equivalent to the valuation $\nu_S$.  
\end{prop}
We omit the straightforward proof which is essentially contained in
\cite{MW}. A proof in a more general context is given in \cite{DLZ}. 

\begin{df}\label{A1d}
  Let $\mu$ be a negative pseudovaluation on a ring $R$.  Let $B$ be a
  finitely generated $R$-algebra. Choose an arbitrary surjection
  $R[T_1, \ldots, T_m] \rightarrow B$ and an arbitrary degree
  valuation $\nu$ on $R[T_1, \ldots T_m] $. Then the quotient
  $\bar{\nu}$ on $B$, is up to linear equivalence independent of these
  choices. We call any negative pseudovaluation in this equivalence
  class admissible.
\end{df}
Let $\mu$ be an admissible pseudovaluation on a finitely generated
$R$-algebra $B$. Let $\nu$ be the pseudovaluation on a polynomial
algebra $B[T_1, \ldots, T_m]$ given by Example 1. Then $\nu$ is
admissible. This it is easily seen, if we write $B$ as a quotient of a
polynomial algebra.

\begin{lemma}\label{A3l}
Let $(R, \mu)$ be a ring with a negative pseudovaluation. Let $A$ be
an $R$-algebra which is finite and free as an $R$-module. Let $\tau$
be an admissible pseudovaluation on $A$.

Choose an $R$-module isomorphism $R^n \cong A$. With respect to this
isomorphism $\tau$ is linearly equivalent to the order function
$\mu^n$ given by (\ref{A11e}). 
\end{lemma}

\begin{proof}
Let $e_1, \ldots, e_n$ be a basis of $A$ as an
$R$-module. Consider the natural surjection:
\begin{displaymath}
   \alpha: R[T_1, \ldots, T_n] \rightarrow A
\end{displaymath}
such that $\alpha (T_i) = e_i$. We have equations:

\begin{displaymath}
   e_ie_j = \sum_{l=1}^{n} c_{ij}^{(l)} e_l, \quad c_{ij}^{(l)} \in R.
\end{displaymath}
We choose a number $d$, such that for all coefficients $c_{ij}^{(l)}$:

\begin{displaymath}
   \mu (c_{ij}^{(l)}) + d \geq 0.
\end{displaymath}
Let $\tilde{\tau}$ be the pseudovaluation (\ref{A9e}) on $R[T_1,
\ldots, T_n]$, such that $\tilde{\tau}(X_i) = -d$. We can take for
$\tau$ the quotient of $\tilde{\tau}$.

Consider an element $a \in A$. We choose a representative of $a$:

\begin{displaymath}
   f = \sum_{k} r_k T_1^{k_1}\cdot \ldots \cdot T_n^{k_n}.
\end{displaymath}
We claim that there is a linear polynomial $f_1$ which maps to $a$,
such that 
\begin{displaymath}
   \tilde{\tau}(f_1) \geq \tilde{\tau}(f).
\end{displaymath}
Indeed, assume that some of the monomials $r_kT_1^{k_1}\cdot \ldots
\cdot T_n^{k_n}$ is divisible by $T_iT_j$. We will pretend in our
notation that $i \neq j$, but the other case is the same. We find an 
equation
\begin{displaymath}
   r_k T^k = \sum_l r_kc_{ij}^{(l)}T^{k(l)},
\end{displaymath} 
where $|k(l)| = |k| -1$. We find for any fixed $l$:
\begin{displaymath}
   \tilde{\tau}(r_kc_{ij}^{(l)}T^{k(l)}) = \mu (r_kc_{ij}^{(l)}) -
   d(|k| -1) \geq \mu (a_k) + \mu(c_{ij}^{(l)}) + d - d|k| \geq \tilde{\tau}(f).
\end{displaymath}

We conclude that
\begin{equation}\label{A10e}
   \tau (a) = \sup \{\tilde{\tau}(f) \;|\; f = r_0 + r_1T_1 + \ldots +
   r_nT_n, \; \alpha(f) = b  \}.
\end{equation}

By construction $a$ has a unique representative
\begin{displaymath}
   g = \sum_{i=1}^n  s_i T_i.
\end{displaymath}
Clearly $\tilde{\tau}$ restricted to linear forms as above is linearly
equivalent to the order function $\mu^n$ defined by (\ref{A11e}). We
need to compare $\tilde{\tau}(g)$ and $\tau(a)$. 

We have a relation in $A$: 
\begin{displaymath}
   1 = \sum_{i=1}^n c_i e_i, \quad c_i \in R.
\end{displaymath}
Given a representative $f$ as in (\ref{A10e}) we find:
\begin{displaymath}
   g = \sum_{i=1}^n (r_i + c_i r_0)T_i.
\end{displaymath}
Then we find:
\begin{displaymath}
\begin{array}{ll}
\tilde{\tau}(g) & =  \min \{\mu(r_i + c_ir_0) - d\} \geq 
\min_{i} \{\min \{\mu(r_i) - d, \mu(c_i) - d +\mu(r_0)  \} \}\\ 
 & \geq  \tilde{\tau}(f) - d',
\end{array} 
\end{displaymath}
where $d'$ is chosen such that $-d' < \mu(c_i) - d$. Since this is
true for arbitrary $f$ we find $\tilde{\tau}(g) \geq \tau(b) - d'$.
Since $\tau $ is the quotient norm we have the obvious inequality 
\begin{displaymath}
   \tau (b) \geq \tilde{\tau}(g).
\end{displaymath} 
This completes the proof. 
\end{proof}
 
\smallskip
 
{\bf Example 2:} Let $\nu$ be a negative pseudovaluation on $A$. Let
$d >0$ a real number.  Then we have defined a pseudovaluation on the
polynomial algebra $A[X]$:
\begin{equation}\label{A2e}
\mu (\sum a_i X^i) = \min \{ \nu (a_i) - id \}.
\end{equation}

Let  $f \in A$, such that $f$ is not nilpotent. 
Then we define a pseudovaluation $\nu'$  on the localization $A_f$ by
taking the quotient under the map: 
\begin{displaymath}
A[X] \rightarrow A_f,
\end{displaymath} 
which sends $X$ to $f^{-1}$. 
As we remarked above $\nu'$ depends only on the linear equivalence
class of $\nu$ on $A$.

Let $z \in A_f$. Consider all possible representations of $z$ in the form:
\begin{equation}\label{A4e}
z = \sum_l a_l/f^l. 
\end{equation}
Then $\nu'$ is the supremum over all these representations of the
following numbers: 
\begin{equation}\label{A5e}
\min_l \{ \nu(a_l) - ld \}.
\end{equation} 
If the supremum is assumed we call the representation  optimal. 

\begin{lemma}\label{LOK7l}
 Let $(A,\nu)$ be a ring with a negative pseudovaluation. Let  $f \in
A$ be a non-zero divisor. Let $\nu'$ be the induced pseudovaluation 
on $A_f$ which is associated to a fixed number $d > 0$. 

We are going to define a function $\tau : A_f \rightarrow \mathbb{R}
\cup \{\infty \}$. For $z \in A_f$ we consider the set of all possible
representations
\begin{equation}\label{LOK1e}
   z = a/(f^m).
\end{equation}  
 We define $\tau (z)$ to be the maximum of the numbers $\nu (a) - md$
for all possible representations (\ref{LOK1e}).

Then there is a real constant $Q > 0$ such that for all $z \in A_f$
\begin{displaymath}
   \nu' (z) \geq \tau (z) \geq Q \nu' (z).
\end{displaymath}  
\end{lemma}

\begin{proof}
The first of the asserted inequalities is
trivial. Consider any representation:
\begin{displaymath}
   z = \sum_{l=0}^{m} u_l/(f^l) \quad \text{such that} \; u_m \neq 0. 
\end{displaymath}
We set 
\begin{displaymath}
   -C = \min_{l} \{\nu(u_l) - ld \}.
\end{displaymath}
We note that this implies that $-C \leq -md$. 
We find a representation of the form (\ref{LOK1e}):
\begin{displaymath}
   z = (\sum_{l=0}^{m} u_l f^{m-l})/(f^m) = a/f^m. 
\end{displaymath}
We find:
\begin{displaymath}
\begin{array}{lcl}
\nu(a) - md & = & \nu(\sum_{l=0}^{m} u_l f^{m-l}) -md \geq 
\min_{l} \{\nu(u_l) + (m-l)\nu(f) -md\}\\
& \geq & \min_{l} \{\nu(u_l) - ld - md  + m\nu(f) + l(d - \nu(f)) \}\\
& \geq & \min_{l} \{-C - C  + m\nu(f) \}.
\end{array}
\end{displaymath}
We have further:
\begin{displaymath}
   m\nu(f) = (-dm) \frac{\nu(f)}{-d} \geq (-C)\frac{\nu(f)}{-d}. 
\end{displaymath}
Together we obtain:
\begin{displaymath}
   \nu (a) - md \geq -C(2 + \frac{\nu(f)}{-d}).
\end{displaymath}
This implies:
\begin{displaymath}
   \tau(a/f^m) \geq (2 + \frac{\nu(f)}{-d})\nu'(z).
\end{displaymath}
\end{proof}

The motivation for the following definition is Lemma \ref{LOK5l} below.

\begin{df}\label{LOK2d}
Let $(A,\nu)$ be a ring with a negative pseudovaluation. 
We say that a non-zero divisor $f \in A$ is localizing with respect to
$\nu$, if there are real numbers $C > 0$ and $D \geq 0$, such that for all 
natural numbers $n$:
\begin{equation}\label{LOK2e}
   \nu(f^n x) \leq C\nu(x) + nD,\quad \text{for all}\; x \in A.
\end{equation}
\end{df}
If $\mu$ is a negative pseudovaluation on $A$, which is linearly
equivalent to $\nu$ then $f$ is localizing with respect to $\nu$, iff
it is  localizing with respect to $\mu$. Indeed any function $\rho$ linearly 
equivalent satisfies an inequality (\ref{LOK2e}).
 
It is helpful to remark that making $C$ smaller we may always arrange
that $D$ is smaller than any given positive number. It is also easy to
see that a unit of the ring $A$ is always localizing. 

Let $A = R[T_1, \ldots, T_d]$ be a polynomial ring over an integral domain
$R$ with a degree valuation $\nu$. Then we have the equation:
\begin{displaymath}
    \nu(f^n x) = \nu(x) + n \nu(f) ,\quad x \in A.
\end{displaymath}
Therefore (\ref{LOK2e}) holds with $C = 1$ and $D =0$.

More generally, let $(B,\mu)$ a ring with a negative pseudovaluation.
We endow $A = B[T]$ with the natural extension $\nu$ of $\mu$ such
that $\nu (T) = -d$. Assume that $f = T^m + a_{m-1}T^{m-1} + \ldots +
a_0$ is a monic polynomial with $a_i \in B$. 
\begin{df}\label{LOK6d}
We say that $f$ is
regular with respect to $T$, if
\begin{equation}\label{LOK6e}
\min_{0 \leq i < m} \{\mu(a_i) -id \} > -md.
\end{equation} 
\end{df}
For a regular polynomial we have $\nu (f) = -md$.
We remark that each monic polynomial $f$ becomes regular for a
suitable choice of $d$.  

\begin{prop}\label{A9p}
Let $f(T) \in B[T]$ be a regular polynomial (\ref{LOK6d}). Then we
have for an arbitrary polynomial $g(T) \in B[T]$ that
\begin{displaymath}
   \nu (f(T) g(T)) = \nu (f(T)) + \nu (g(T)).
\end{displaymath}
In particular any monic polynomial in $B[T]$ is localizing.
\end{prop}

\begin{proof} 
We write 
\begin{displaymath}
   g = \sum_{k=0}^{n} b_k T^k.
\end{displaymath}
Let $k_0$ be the largest index such that $\nu(g) = \nu(b_{k_0}) - k_0
d$. $fg$ contains the monomial 
\begin{displaymath}
   (b_{k_0} + b_{k_0+1}a_{m-1} + \ldots )T^{m+k_{0}}.
\end{displaymath}
We find by (\ref{LOK6e}) that
\begin{displaymath}
   \mu (b_{k_0+i}a_{m-i}) \geq \mu (b_{k_0}) + \mu (a_{m-i}) \geq \mu
   (b_{k_0}) -id.  
\end{displaymath}
On the other hand we have by the choice of $k_{0}$ that 
\begin{displaymath}
   \mu (b_{k_0}) - k_0 d < \mu (b_{k_0 + i}) - (i+k_0)d.
\end{displaymath}
This proves that  $\mu (b_{k_0+i}a_{m-i}) > \mu (b_{k_0}).$ Therefore
we obtain that
\begin{displaymath}
   \mu (b_{k_0} + b_{k_0+1}a_{m-1} + \ldots ) = \mu (b_{k_0}).
\end{displaymath} 
This shows the inequality:
\begin{displaymath}
   \nu (fg) \leq \nu(b_{k_0}) - d(m+k_0) = \nu(g) -md = \nu (g) + \nu
   (f). 
\end{displaymath}
The opposite inequality is obvious.  The last assertion follows
because any monic polynomial is regular for a suitable chosen $d$. 
\end{proof}

\begin{prop}
Assume that $f = T^m + a_{m-1}T^{m-1} + \ldots + a_0 \in B[T]$ is a 
polynomial which is regular with respect to $T$. 
Each $z \in A_f$ has a unique representation:
\begin{equation}\label{A6e}
z = \sum_l u_l/f^l, \qquad u_l \in B[T],
\end{equation}
where $u_l$ is for $l > 0$ a polynomial of degree strictly less than $m = \deg f$.
Then the representation (\ref{A6e}) is optimal (compare (\ref{A5e})).
\end{prop}

\begin{proof} The first assertion follows from the euclidian division. 
Consider any other representation 
\begin{displaymath}
   z = \sum_i v_i/f^i, \qquad v_i \in B[T],
\end{displaymath}
 Assume that $ n = \deg v_i \geq m$ for some
$i >0$. Let $c \in B$ be the highest coefficient of the polynomial $v_i$ and set $t = n - m$.
Then we conclude:
\begin{displaymath}
\nu (c T^tf) \geq \mu(c) + \nu (T^t) + \nu (f) =   \mu(c) + \nu (T^t) + \nu (T^m)
= \nu (c T^n) \geq \nu (v_i).
\end{displaymath}
We write:
\begin{displaymath}
v_i/f^i = ((v_i - cT^tf)/f^i) - (cT^t/f^{i-1}).
\end{displaymath}
If we insert this in the representation (\ref{A4e}) the number (\ref{A5e}) becomes bigger
because:
\begin{displaymath}
\nu ((v_i - cT^t f) \geq \nu (v_i), \qquad \nu (cT^t) \geq \nu (cT^n) \geq \nu (v_i). 
\end{displaymath}
Continuing this process proves the lemma. 
\end{proof}

The last Proposition applies in particular to a polynomial ring over a
field $A = K[T_1, \dots, T_d]$ with the standard degree valuation. By
Noether normalization any polynomial becomes regular with respect to
some variable after a coordinate change.

\begin{prop}\label{LOK2p}
Let $(A,\nu)$ be a ring with a pseudovaluation. Let $f,g \in A$.
Then $fg$ is localizing iff $f$ and $g$ are localizing.
\end{prop}
{\bf Proof}: Assume $fg$ is localizing. Then we find an inequality:
\begin{displaymath}
   \nu (f^ng^nx) \leq C \nu (x) +nD.
\end{displaymath}
On the other hand we have the inequality:
\begin{displaymath}
   n\nu(f) + \nu (g^nx) \leq \nu (f^ng^nx). 
\end{displaymath}
This shows that:
\begin{displaymath}
 \nu (g^nx) \leq C \nu(x) + n(D - \nu (f)).  
\end{displaymath}
We leave the opposite implication to the reader.

\begin{prop}\label{A5p}
Let $(A,\nu)$ be a ring with a pseudovaluation. Assume that $A$
is an integral domain, such that each non-zero element of $A$ is
localizing. Let $A \rightarrow B$ be a finite ring homomorphism 
such that $B$ is a free $A$-module. Let $\mu$ be an admissible 
pseudovaluation on $B$. Then any nonzero divisor in $B$ is localizing 
with respect to $\mu$.
\end{prop}
\begin{proof}
 We choose an isomorphism of $A$-modules: $A^r \cong B$.
By Lemma \ref{A3l} the order function $\nu^r$ on $A^r$ is linearly
equivalent to an admissible pseudovaluation on $B$. Let $f \in A$, $f
\neq  0$. Then an inequality (\ref{LOK2e}) holds. It follows that for 
each $z \in A^r$: 
\begin{displaymath}
   \nu^r (f^n z) \leq C\nu^r (z) + nD. 
\end{displaymath}
This shows that $f$ is localizing in $B$. More generally consider a
non-zero divisor $b \in B$. Consider an equation of minimal degree:
\begin{displaymath}
   b^t + a_{t-1}b^{t-1} + \ldots + a_1 b + a_0 = 0, \qquad a_i \in A.
\end{displaymath}
Then $a_0 \neq 0$ and therefore localizing. But $a_0$ is a multiple of
$b$ in the ring $B$. Therefore $b$ is localizing in $B$ by Proposition
\ref{LOK2p}. 
\end{proof}

\begin{cor}\label{A5c}
Let $X \rightarrow \Spec K$ be a smooth scheme over a field $K$ of
characteristic $p$. Then
any point of $X$ has an affine neighbourhood $\Spec A$, such that any  
non-zero element in $A$ is localizing.
\end{cor}
\begin{proof} This is immediate from a result of \cite{K1} which says
that each point admits a neighbourhood which is finite and \'etale 
over an affine space $\mathbb{A}^n_{K}$. \end{proof}

Let us assume that $f \in A$ is localizing with constants $C,D$ given
by (\ref{LOK2e}). Then we will assume that the constant $d$ used in
the definition of $\nu'$ on $A_f$ is bigger than $D$. This 
can be done with no loss of generality because the equivalence class
of $\nu'$ doesn't depend on $d$.

\begin{prop}\label{A6p}
Let $(A,\nu)$ be a ring with a negative pseudovaluation.
Let $f \in A$ be a localizing element.  Each $z \in A_f$ has a unique
representation 
\begin{displaymath}
   z = a/f^m, \quad \text{where} \; a \in A, \; f \nmid a. 
\end{displaymath}
We define a real valued function $\sigma$ on $A_f$:
\begin{displaymath}
   \sigma (z) = \nu (a) - md.
\end{displaymath}
Then there exists a real constant $E > 0$, such that:
\begin{displaymath}
   \nu' (z) \geq \sigma (z) \geq E \nu'(z). 
\end{displaymath}
In particular, the restriction of $\nu'$ to $A$ is linearly equivalent
to $\nu$. 
\end{prop}
\begin{proof} By Lemma \ref{LOK7l} it suffices to show the last inequality with
$\nu'$ replaced by $\tau$. All representations 
(\ref{LOK1e}) of $z$ are of the form:
\begin{displaymath}
   af^r/f^{m+r}.
\end{displaymath}
Since $f$ is localizing there are real numbers $1 > C>0$ and $D\geq 0$
such that:
\begin{displaymath}
\begin{array}{lcl}
\nu(af^r) - (m+r)d & \leq & C\nu(a) + rD -md -rd\\
& \leq & C(\nu(a) - md) + (D-d)r \leq C\sigma (z)  + (D-d)r.
\end{array}
\end{displaymath}
We may assume that $d \geq D$. Then the inequality above implies
\begin{displaymath}
   \tau (z) \leq C \sigma (z).
\end{displaymath}
\end{proof}

\begin{cor}\label{A13c}
Let $(B,\mu)$ be an integral domain with a pseudovaluation $\mu$.
Assume that each non-zero element is localizing. We endow $B[T]$
with a pseudovaluation of Example 1.

Then each non-zero element in $B[T]$ is localizing.
\end{cor}
\begin{proof}
 Clearly each $b \in B$, $b\neq 0$ is localizing in
$B[T]$. By the Proposition it suffices to find for a given $f \in
B[T]$ an element $b \in B$, such that $f$ is localizing in $B_b[T]$.
By the remark preceding Proposition \ref{A9p} we may assume that $f$ is a 
regular polynomial. Then we can apply this proposition. 
\end{proof}

The following corollary would allow to prove Corollary \ref{A5c} 
more generally by considering standard \'etale neighbourhoods instead 
of Kedlaya's result. 

\begin{cor}
Let $(A,\nu)$ be a noetherian ring  with a negative pseudovaluation.
Let $a,f \in A$ be two localizing elements. Then $a$ is localizing in 
$A_f$. 
\end{cor}
\begin{proof} By the Lemma of Artin-Rees there is a natural number $r$, 
such that for $m \geq r$
\begin{displaymath}
   ax \in f^m A \quad \text{implies} \quad x \in f^{m-r}A.
\end{displaymath}
Assume that $x \in A$, but $x \notin fA$. Then we conclude that for
each $n \in \mathbb{N}$
\begin{displaymath}
   a^n x \in f^m A \quad \text{implies} \quad m \leq nr.
\end{displaymath}
Consider a reduced fraction $x/f^m \in A_f$. To show that $a$ is
localizing it suffices to find an estimation for
\begin{displaymath}
   \sigma (a^n (x/f^m)),
\end{displaymath} 
where $\sigma$ is the function of Proposition \ref{A6p}:
\begin{displaymath}
   \sigma (x/f^m) = \nu (x) - md.
\end{displaymath}
By the remarks above we may write with $y \notin fA$:
\begin{displaymath}
   \frac{a^n x}{f^m} = \frac{yf^s}{f^m}, \quad s \leq nr.
\end{displaymath}
Using this equation we obtain:
\begin{displaymath}
\begin{array}{ll}
 \nu(y) & \leq \nu (yf^s) - \nu (f^s) \leq \nu (a^nx) - s\nu(f)\\ 
 & \leq C\nu (x) + nD - nr \nu (f).
\end{array}
\end{displaymath}
Here $C \leq 1, D$ are positive real constants, which exists because
$a$ is localizing in $A$.

Now it is easy to give an estimation for
\begin{displaymath}
    \sigma (a^n (x/f^m))  = \sigma (yf^s/f^{m}).
\end{displaymath}
We omit the details. 
\end{proof}

We reformulate Proposition \ref{A6p} in the case where $A = R[T_1,
\ldots, T_d]$ is a  polynomial algebra over an integral domain $R$
with the 
standard degree valuation $\nu$. It extends to a valuation on the
quotient field of $A$ which we denote by $\nu$ too. 
Let $f \in A$ be a non-zero element. We define $\nu'$ on $A_f$
associated to $d >0$ as before (\ref{A5e}).

We define a second pseudovaluation $\mu $ on the ring $A_f$ as follows.
Let $\vartheta (z)$ be the smallest integer $n\geq 0$ such that $f^n z
\in A$. We set: 
\begin{equation}\label{A7e}
\mu (z) = \min \{\nu' (z), - d\vartheta (z) \}  
\end{equation}  

\begin{prop}
Let $A$ be a polynomial ring with the standard degree valuation $\nu$.
Let $f \in A$ be a non constant polynomial. Let us define
pseudovaluations $\nu'$ resp. $\mu$ on $A_f$ by the formulas
(\ref{A5e}) resp. (\ref{A7e}). Then there are constants 
$Q_1$ and $Q_2$, such that
\begin{displaymath}
Q_1\mu \geq \nu' \geq Q_2 \mu.
\end{displaymath}
\end{prop}
\begin{proof} We write an element $z \in A_f$ as a reduced fraction
\begin{displaymath}
   z = (a/f^m),
\end{displaymath}
such that $m = \vartheta (z)$. 
By Proposition \ref{A6p} it is enough to compare $\mu$ with the
function $\sigma$. The inequality $\sigma (z) \leq \mu (z)$ is
obvious. We show that for a sufficiently big number $C >1$:
\begin{displaymath}
   C\mu (z) \leq \nu (a) - md.
\end{displaymath}
This is obvious if $-Cmd \leq -md +\nu (a)$. Therefore we can make
that assumption:
\begin{displaymath}
   -(C-1)md \geq \nu (a).
\end{displaymath}
We have to find $C$ such that the following inequality is satisfied:
\begin{displaymath}
   C(\nu (a) - m\nu (f)) \leq \nu (a) - md.
\end{displaymath}
We have by assumption: 
\begin{displaymath}
   (C-1) \nu (a) \leq - (C-1)^2 md.
\end{displaymath}
Therefore it suffices to show that for big $C$:
\begin{displaymath}
  - (C-1)^2 md \leq m(C\nu(f) -d). 
\end{displaymath}
But this is obvious.
\end{proof}

\section{Overconvergent Witt vectors}

Let us fix a prime number $p$.
We are going to introduce the ring of overconvergent Witt vectors.
Let $A$ be a ring with a proper pseudovaluation $\nu$. We assume that $pA =0$. 

Let $W(A)$ be the ring of Witt vectors. For any Witt vector
\begin{displaymath}
\alpha = (a_0, a_1, a_2, \ldots  ) \in W(A)
\end{displaymath} 
we consider the following set $\mathcal{T}(\alpha)$ in the $x-y$-plane:
\begin{displaymath}
(p^{-i} \nu(a_i),i), \qquad \nu(a_i) \neq \infty.
\end{displaymath} 
For $\varepsilon, c \in \mathbb{R}, \; \varepsilon > 0$ we consider the half plane:
\begin{displaymath}
H_{\varepsilon, c} = \{ (x,y) \in \mathbb{R}^2 \;|\; y \geq -\varepsilon x + c \}.
\end{displaymath} 
Moreover we consider for all $c \in \mathbb{R}$ the half plane:
\begin{displaymath}
H_c =  \{ (x,y) \in \mathbb{R}^2 \;|\; x \geq c \}.
\end{displaymath} 
Let $\mathcal{H}$ the set of all half planes of the two different types above. 
We define the Newton polygon $\NP(\alpha)$:
\begin{displaymath}
\NP(\alpha) = \bigcap_{H \in \mathcal{H}, \mathcal{T}(\alpha) \in H} H.
\end{displaymath} 

\begin{df}
We say that a Witt vector $\alpha$  has radius of convergence $\varepsilon > 0$, if there is a 
constant $c \in \mathbb{R}$, such that
\begin{displaymath}
i \geq - \varepsilon p^{-i}\nu(a_i) + c.
\end{displaymath} 
We denote the set of these Witt vectors by $W^{\varepsilon} (A)$.
\end{df}
Equivalently one may say that the Newton polygon $\NP(\alpha)$ lies
above a line of slope $-\varepsilon$.

We define the Gauss norm $\gamma_{\varepsilon} : W(A) \rightarrow \mathbb{R}$: 
\begin{equation}\label{A8e}
\gamma_{\varepsilon}(\alpha) = \inf \{i + \varepsilon p^{-i}\nu(a_i)\}
\end{equation} 
Convergence of radius $\varepsilon > 0$ means that
$\gamma_{\varepsilon}(\alpha) \neq - \infty$. We will denote the set of
Witt vectors of radius of convergence $\varepsilon$ by $W^{\varepsilon}
(A)$.

\begin{prop}\label{A10p}
  Let $(A, \nu)$ a ring with a proper pseudovaluation, such that $pA =
  0$. Then for any $\varepsilon > 0$ the Gauss norm $\gamma_{\varepsilon}$
  is a pseudovaluation on $W(A)$. In particular $W^{\varepsilon} (A)$ is
  a ring. 

If we assume moreover that $\nu$ is a valuation, we have the equality
for arbitrary $\xi, \eta \in W^{\varepsilon}(A)$:
\begin{equation}\label{A25e}
\gamma_{\varepsilon}(\xi \eta) = \gamma_{\varepsilon}(\xi) + \gamma_{\varepsilon}(\eta).
\end{equation} 
\end{prop}

\begin{proof} Clearly we may assume $\varepsilon = 1$. We set $\gamma =
\gamma_1$. The first two requirements of Definition \ref{A11d} are
clear. Consider two Wittvectors:
\begin{displaymath}
   \xi = (a_0, a_1, \ldots ) \in W(A), \quad \eta = (b_0, b_1, \ldots)
   \in W(A). 
\end{displaymath}
We begin to show the inequality:
\begin{displaymath}
   \gamma (\xi + \eta) \geq \min \{ \gamma (\xi), \gamma (\eta) \}.
\end{displaymath} 
We may assume that there is $g \in \mathbb{R}$, such that 
\begin{displaymath}
   i + p^{-i}\nu(a_i) \geq g, \quad  i + p^{-i}\nu(b_i) \geq g.
\end{displaymath}
We write
\begin{displaymath}
   \xi + \eta = (s_0, s_1, \ldots ).
\end{displaymath}
Let $S_m$ be the polynomials, which define the addition of the Witt vectors:
\begin{displaymath}
   s_m = S_m (a_0, \ldots, a_m, b_0, \ldots, b_m).
\end{displaymath}
We know that $S_m$ is a sum of monomials
\begin{displaymath}
   M = \pm a_0^{e_0}\cdot \ldots \cdot a_m^{e_m}b_0^{f_0}\cdot \ldots
   \cdot b_m^{f_m},
\end{displaymath}
such that 
\begin{displaymath}
   \sum_{i=0}^{m}p^i e_i + \sum_{i=0}^{m}p^i f_i = p^m.
\end{displaymath}
We have to show that $p^{-m}\nu(m) + m \geq m$. We compute:
\begin{displaymath}
\begin{array}{l}
 p^{-m} \nu (M) + m \\[2mm]
\geq p^{-m}(\sum_{i=0}^{m} e_i\nu(a_i) + \sum_{i=0}^{m} f_i \nu (b_i)\\[2mm]
+ \sum_{i=0}^{m} p^ie_i m + \sum_{i=0}^{m} p^if_i m)\\[2mm]
\geq p^{-m}(\sum_{i=0}^{m}p^ie_i(p^{-i}\nu(a_i) +i) +
\sum_{i=0}^{m}p^if_i(p^{-i}\nu(b_i) +i))\\[2mm]
\geq p^{-m}(\sum_{i=0}^{m}p^ie_i g + \sum_{i=0}^{m}p^if_i g) \geq g.  
\end{array}
\end{displaymath}
This proves the fourth requirement of Definition \ref{A11d}.

Next we prove the inequality:
\begin{equation}\label{A29e}
   \gamma (\xi \eta) \geq \gamma (\xi) + \gamma (\eta).
\end{equation}
By the inequality already shown we are reduced to the case
\begin{displaymath}
   \xi = ~^{V^i} [a], \quad \text{and} \; \eta = ~^{V^j}[b].
\end{displaymath} 
Since by assumption $F$ and $V$ commute on $W(A)$ we find 
\begin{displaymath}
   \xi \eta = ~^{V^{i+j}}[a^{p^j} b^{p^i}].
\end{displaymath}
We obtain:
\begin{equation}\label{A23e}
\begin{array}{ll}
\gamma(\xi \eta) & = \frac{\nu(a^{p^j}b^{p^i})}{p^{i+j}} + i + j\\[2mm]  
& \geq \frac{p^j\nu(a) + p^i \nu(b)}{p^{i+j}} + i + j = \gamma (\xi) +
\gamma (\eta). 
\end{array}
\end{equation}
This proves that $\gamma$ is a pseudovaluation. 

Finally we prove the equality (\ref{A25e}) if $\nu$ is a valuation. We
remark that (\ref{A23e}) is an equality in this case. From this we
obtain (\ref{A25e}) in the case where
\begin{displaymath}
   \xi = ~^{V^i}[a] + \xi_1,  \quad \eta = ~^{V^j}[b] + \eta_2,
\end{displaymath}
where $\xi_1 \in V^{i+1}W(A)$, $\eta \in V^{j+1}W(A)$ and 
\begin{displaymath}
   \gamma (\xi_1) \geq \gamma  (~^{V^i}[a]), \quad \gamma (\eta_2) \geq
   \gamma (~^{V^j}[b]).
\end{displaymath}
Next we consider the case if there are $i$ and $j$ such that
\begin{displaymath}
   p^{-i}\nu(a_i) + i = \gamma (\xi), \quad  p^{-j}\nu(b_j) + j = \gamma (\eta).
\end{displaymath}
We assume that $i$ and $j$ are minimal with this property. Then we
write
\begin{displaymath}
\begin{array}{lll}
  \xi & = (a_0, \ldots, a_{i-1}, 0, \ldots) + \xi_1 & = \xi' + \xi_1\\
  \eta & = (b_0, \ldots, b_{j-1}, 0, \ldots) + \eta_1 & = \eta' + \eta_1.
\end{array}
\end{displaymath}
Then we have by our choice:
\begin{displaymath}
   \gamma (\xi') > \gamma (\xi) = \gamma (\xi_1), \;
   \gamma (\eta') > \gamma (\eta) = \gamma (\eta_1).
\end{displaymath}
By the case already treated we have $\gamma(\xi_1 \eta_1) = \gamma
(\xi_1) + \gamma (\eta_1)$. Then we obtain:
\begin{equation}\label{A27e}
   \gamma (\xi \eta) = \gamma (\xi_1 \eta_1 + \xi_1 \eta' + \xi'
   \eta_1 + \xi' \eta') \geq \min \{\gamma(\xi_1 \eta_1) +
   \gamma(\xi_1 \eta')  + \gamma(\xi'\eta_1) + \gamma(\xi' \eta') \}.
\end{equation}
But by the inequality (\ref{A29e}) this minimum is assumed only for
$\gamma(\xi_1 \eta_1)$ and therefore (\ref{A27e}) is an equality. 

Finally if $i$ and $j$ as above don't exist this becomes true if we
replace $\varepsilon$ by any $\delta$ which is a little smaller. If
$\delta$ approaches $\varepsilon$ we obtain the result. 
\end{proof}

We have the formulas:
\begin{equation}\label{A21e}
 \begin{array}{rcl}
\gamma_{\varepsilon} (~^V\alpha) & = & 1 + \gamma_{\varepsilon/p}(\alpha)\\
\gamma_{\varepsilon} (~^F\alpha) & \geq & \gamma_{p\varepsilon}(\alpha)\\
\gamma_{\varepsilon} (p) & = & 1.
\end{array}
\end{equation} 

\begin{df}\label{A10d}  
The union of the rings $W^{\varepsilon}(A)$ for $\varepsilon > 0$ is called
the ring of overconvergent Witt vectors $W^{\dagger}(A)$.
\end{df}

\begin{cor}\label{A10c}
Let $\alpha \in W^{\dagger}(A)$ and let $\delta > 0$ a real
number. Then there is an $\varepsilon > 0$  
such that $\gamma_{\varepsilon}(\alpha) > - \delta$.
\end{cor} 
\begin{proof} Take some negative line of slope $-\tau$ below
the Newton polygon of $\alpha$.  If this line does not meet the
negative $x$-axis we conclude that $\gamma_{\tau}(\alpha) \geq 0$.  In
the other case we rotate the line around the intersection point to
obtain the desired slope $-\varepsilon$. \end{proof}

\smallskip

We will from now on assume that the pseudovaluation $\nu $ on $A$ is
negative. By proposition \ref{A10p} this is a subring. Then we have:
\begin{displaymath}
W^{\delta}(A) \subset W^{\varepsilon}(A), \quad \mbox{if} \; \delta > \varepsilon.
\end{displaymath}

 The ring $W^{\dagger}(A)$ does not change if we replace
$\nu$ by a linearly equivalent pseudovaluation.
More generally let $f: A \rightarrow \mathbb{R} \cup
\{\infty \}$ be any function which is linearly equivalent to $\nu$.
Then a Witt vector $(x_0, x_1, \ldots ) \in W(A)$ is overconvergent
with respect to the $\nu$, iff there is an $\varepsilon > 0$ and a
constant $C \in \mathbb{R}$ such that for all $i\geq 0$.
\begin{displaymath}
   i + p^{i}f(x_i)\varepsilon \geq -C.
\end{displaymath} 

With the notation of Definition \ref{A1d} let $A$ be a finitely
generated algebra over $(R,\mu)$. Any admissible
pseudovaluation on $A$ leads to the same ring $W^{\dagger}(A)$. 
Let $\alpha : A \rightarrow B$ be a homomorphism of finitely generated
$R$-algebras. Then the induced homomorphism on the rings of Witt
vectors respects overconvergent Witt vectors:
\begin{equation}\label{A15e}
   W(\alpha): W^{\dagger}(A) \rightarrow W^{\dagger}(B).
\end{equation}
This is seen by choosing a diagram
\begin{displaymath}
   \xymatrix{
R[T_1, \ldots, T_n] \ar[r] \ar[d] & R[T_1, \ldots, T_n, S_1, \ldots,
S_m] \ar[d]\\
A \ar[r] & B.\\
}
\end{displaymath}

On the truncated Witt vectors we consider the functions $\gamma_{\varepsilon}[n]$:
\begin{displaymath}
\gamma_{\varepsilon}[n]: W_{n+1}(A) \rightarrow \mathbb{R} \cup \{ \infty \}
\end{displaymath} 
\begin{displaymath}
\gamma_{\varepsilon}[n](\alpha) = \min \{i + \varepsilon p^{-i}\nu(a_i) \;|\; i\leq n \}.
\end{displaymath}
This is the quotient of $\gamma_{\varepsilon}$ under the natural map
$W(A) \rightarrow W_{n+1} (A)$ in the sense of (\ref{A19e}). We
conclude that $\gamma_{\varepsilon}[n]$ is a proper pseudovaluation.

The following is obvious:
Let $\sum_{m=0}^{\infty} \alpha_m$ be an infinite sum of Witt vectors $\alpha_m \in W(A)$,
which converges in the $V$-adic topology to $\sigma \in A$.
 Let $\varepsilon > 0$ and $C \in \mathbb{R}$, such that
\begin{displaymath}
\gamma_{\varepsilon} (\alpha_m) \geq C. 
\end{displaymath}
Then $\sigma$ is overconvergent, and we have $\gamma_{\varepsilon} (\sigma) \geq C$.

More generally we can consider families of pseudovaluations $\delta_{\varepsilon}[n]$ of $W(A)$
 which are indexed by real numbers $\varepsilon >0$ and $n \in \mathbb{N}
 \cup \{\infty\}$. We write $\delta_{\varepsilon} = \delta_{\varepsilon}
 [\infty]$. We require that  
\begin{displaymath}
\begin{array}{lcll}
\delta_{\varepsilon_1}[n] & \geq & \delta_{\varepsilon_2}[n] & \quad \varepsilon_1
\leq \varepsilon_2.\\
\delta_{\varepsilon}[n] & \geq & \delta_{\varepsilon_2}[m] & \quad n \leq m.
\end{array}
\end{displaymath}

\begin{df}
Two families $\delta_{\varepsilon}[n]$ and $\delta'_{\varepsilon}[n]$ as above are called equivalent,
if there are constants $c_1, c_2, d_1, d_2 \in \mathbb{R}$, where $c_1 > 0, c_2 > 0$, such
that for sufficiently small $\varepsilon$ the following inequalities hold:
\begin{displaymath}
\begin{array}{ccc}
\delta_{c_1\varepsilon}[n] & \geq & \delta'_{\varepsilon}[n] - d_1\\
\delta'_{c_2\varepsilon}[n] & \geq & \delta_{\varepsilon}[n] - d_2.
\end{array}
\end{displaymath} 
\end{df}

Let $\nu$ and $\nu'$ be negative pseudovaluations on $A$, which are
linearly equivalent. Then the families $\gamma_{\varepsilon}$ and
$\gamma'_{\varepsilon}$ of Gauss norms defined by (\ref{A8e}) are
equivalent.

We obtain from Lemma \ref{A3l}.
\begin{prop}\label{A3p}
Let $(R,\mu)$ be a ring with a negative pseudovaluation. Let $A$ be an
$R$-algebra which is free as an $R$-module. Let $\tau$ an admissible 
pseudovaluation on $A$ given by Proposition \ref{A1p}.

We transport $\mu^n$ to
$A$ by an isomorphism $R^n \cong A$. Then a Witt vector $(a_0, a_1,
\dots ) \in W(A)$ is overconvergent with respect to $\tau$, iff there is
an $\varepsilon > 0$ and a constant $C \in \mathbb{R}$ such that:
\begin{displaymath}
   i + p^{-i}\mu^n(a_i) \geq -C.
\end{displaymath}
In particular a Witt vector $\underline{r} = (r_0, r_1, \ldots ) \in W(R)$ is
overconvergent iff its image in $W(A)$ is overconvergent. 
\end{prop}
\begin{proof} Only the last sentence needs a justification. Assume
$\underline{r}$ is overconvergent in $A$. By the first part of the
Proposition this means the following:

Let $e_i$ be a basis of the $R$-module $A$. we write:
\begin{displaymath}
   1 = \sum_{m} c_m e_m, \quad c_m \in R.
\end{displaymath}
Then overconvergence means that there are constants $\varepsilon >0$ and
$C \in \mathbb{R}$, such that for $1 \leq m \leq n$ and $i \geq 0$
\begin{displaymath}
   i + p^{-i}\mu(c_mr_i)\varepsilon \geq C.
\end{displaymath}
By Cohen-Seidenberg it is clear that $c_m$ generate the unit ideal in
$R$:
\begin{displaymath}
   1 = \sum_m c_m u_m.
\end{displaymath}
This gives
\begin{displaymath}
   \mu(r_i) \geq \min \{\mu(c_m r_i) + \mu(u_i)  \} \geq \min
   \{\mu(c_m r_i) \} - C'.
\end{displaymath}
for some constant $C'$, which depends only on the elements $u_m$.
Therefore we see that $\underline{r} \in W^{\dagger}(R)$. 
We leave the inclusion $W^{\dagger}(R) \subset W^{\dagger}(A)$ to the
reader. \end{proof}

\begin{lemma}\label{LOK5l}
Assume that $f \in A$ is localizing. Let $\underline{c} \in W(A)$ be a Witt
vector such that $\underline{c} \in W^{\dagger}(A_f)$. Then
$\underline{c} \in W^{\dagger} (A)$. 
\end{lemma}
\begin{proof}
We write $\underline{c} = (c_0, c_1, c_2, \ldots )$, where $c_i \in A$. 
By Lemma \ref{LOK7l} we find representations $c_i = a_i/f^{m_i}$, and 
real numbers $\varepsilon > 0$ and $U$, such that 
\begin{displaymath}
   i + p^{-i}\varepsilon (\nu (a_i) - m_i d) \geq -U.
\end{displaymath}
Since $f$ is localizing we find:
\begin{displaymath}
   \nu (a_i) = \nu (f^{m_i} c_i) \leq C \nu (c_i) + m_i D,
\end{displaymath}
and therefore
\begin{displaymath}
   -U \leq i + p^{-i}\varepsilon (C \nu(c_i) + m_i(D-d)) = i +
   p^{-i}\varepsilon C \nu (c_i) + p^{-i}\varepsilon m_i (D-d).
\end{displaymath} 
By our choice $D < d$ the last summand is not positive. This shows 
that $\underline{c} \in W^{\dagger}(A)$.
\end{proof}

\begin{prop}\label{LOK17p}
Let $(A,\nu)$ be an integral domain with a negative pseudovaluation, such
that any non-zero element is localizing. Let $\alpha : A \rightarrow B$
be an injective ring homomorphism of finite type, which is generically 
finite. Then we have:
\begin{displaymath}
   W(A) \cap W^{\dagger}(B) = W^{\dagger}(A).
\end{displaymath} 
\end{prop}
\begin{proof} Indeed, we find an element $c \in A$, $ c\neq 0$ such
that $A_c \rightarrow B_c$ is finite, and $B_c$ is a free $A_c$-module.
Clearly it suffices to show the Proposition if we replace $B$ by
$B_c$. We consider the maps $A \rightarrow A_c \rightarrow B_c$ and
apply the last Lemma and Proposition \ref{A3p}. 
\end{proof}

\begin{prop}\label{LOK18p}
Let $A \rightarrow B$ be a smooth morphism of finitely generated
algebras over a field $K$ of characteristic $p$. We endow them with 
admissible pseudovaluations.
Then we have
\begin{displaymath}
    W(A) \cap W^{\dagger}(B) = W^{\dagger}(A).
\end{displaymath}
\end{prop}
\begin{proof} By \cite{DLZ} $W^{\dagger}$ is a sheaf in the Zariski-topology. 
Therefore the question is local on $\Spec A$. We therefore may assume
by Corollary \ref{A13c} that any non-zero element of $A$ is
localizing. Obviously the 
question  is local on $\Spec B$. By the definition of smooth we may 
therefore assume that the morphism factors
\begin{displaymath}
   A \rightarrow A[T_1, \ldots, T_d] \rightarrow B,
\end{displaymath}
where the last arrow is \'etale and in particular generically finite.
We show the Proposition for both arrows separately. 

We know by the remark after Definition \ref{A1d} that there is an
admissible pseudovaluation on $A[T_1, \ldots, T_d]$, whose
restriction to $A$ is an admissible pseudovaluation. This shows the
assertion for the first arrow.

For the second arrow we use Proposition \ref{LOK17p}.  It is enough to
show that any element in  $C = A[T_1, \ldots, T_d]$ is localizing. But
this is Corollary \ref{A13c}. 
\end{proof}

Let $R$ be an integral domain and endow it with the trivial
valuation. Consider on 
the polynomial ring $A = R[T_1, \ldots, T_d]$ a degree valuation 
$\nu$, such that $\nu(T_i) = -\delta_i < 0$. Let $\gamma_{\varepsilon}$
be the associated Gauss norms on $W(A)$. In the following we need the
dependence on $\delta$. Therefore we set:
\begin{displaymath}
   \gamma^{(\delta)} = \gamma_1, \quad \text{and then} \;  
\gamma^{(\varepsilon\delta)} = \gamma_{\varepsilon}.
\end{displaymath}

Let us denote by $[1,d]$ the set of natural numbers between $1$ and
$d$. A weight $k$ is a function $k: [1,d] \rightarrow \mathbb{Z}_{\geq
  0}[1/p]$. Its values are denoted by $k_i$. 
The denominator of $k$ is the smallest number $u$ such
that $p^uk$ takes values in $\mathbb{Z}$. We set $\delta(k) = k_1\delta_1 +
\ldots + k_d\delta_d$. We write $X_i = T_i$ for the Teichm\"uller
representative and we set $X^k = X_1^{k_1}\cdot \ldots \cdot
X_d^{k_d}$. 

 By \cite{LZ} any element $\alpha \in W(A)$ has a unique
expansion:
\begin{equation}\label{A12e}
\alpha = \sum_k \xi_k X^k, \quad \xi_k \in V^uW(R).
\end{equation} 
Here $u$ denotes the denominator of $k$. This series is
convergent in the $V$-adic topology, i.e. for a given $m \in
\mathbb{N}$ we have $\xi_k \in V^mW(R)$ for almost all $k$.  

For $\xi \in W(R)$ we define:
\begin{displaymath}
\ord_V \xi = \min \{m \;|\; \xi \in V^mW(R)\}.
\end{displaymath}
\begin{prop}
The Gauss norm of $\gamma^{(\delta)}$ is given by the following formula:
\begin{equation}
\gamma^{(\delta)}(\alpha) = \inf \{\ord_V \xi_k -  \delta(k) \}
\end{equation}
and the truncated Gauss norm is given by:
\begin{equation}\label{A17e}
\begin{array}{ll}
   \gamma^{(\delta)}[n](\alpha) & = \min \{\infty , \ord_V \xi_k -
   \delta(k) \;|\; \xi_k \notin V^{n+1}W(R) \}\\
&  = \min_{k} \{ \gamma^{(\delta)}[n](\xi_k X^k) \}. 
\end{array}
\end{equation}
\end{prop}
\begin{proof} It is enough to show the equation (\ref{A17e}). The
formula is obvious if $\alpha = \xi_k X^k$ for a particular $k$. 
This implies (\ref{A17e}) if the minimum is attained exactly once on
the right hand side. 

Let $\delta^{(l)} \in \mathbb{R}^d_{>0}$, $l \in \mathbb{N}$ be a
sequence which converges to 
the given $\delta$. We denote by $\gamma^{(l)}[n]$ the 
truncated Gauss norm on $W_{n+1} (A)$ associated to numbers
$\delta^{(l)}$. We easily see that
\begin{displaymath}
   \lim_{l\rightarrow \infty} \gamma^{(l)}[n](\alpha) = 
   \gamma^{(\delta)}[n] (\alpha).
\end{displaymath}
Clearly the right hand side of (\ref{A17e}) is also continuous with
respect to $\delta$. Therefore it suffices for the proof to construct
a sequence $\delta^{(l)}$ such that for each $l$ the minimum
\begin{displaymath}
   \min \{ \gamma^{(l)}[n](\xi_k X^k) \}
\end{displaymath}
is assumed exactly once. This is the case for $\alpha \neq 0$.
Indeed on the right hand side of (\ref{A17e}) all but finitely many 
$\gamma_{\varepsilon}[n](\xi_k X^k)$ are equal to $\infty$. We denote by
$g$ the smallest of these values and by $g_1$ the next greater value
which may be $\infty$. Let $T$ be the set of weights where the 
value $g$ is assumed. 

The set of linear functions $\eta: \mathbb{R}^d \rightarrow
\mathbb{R}$ such that
\begin{displaymath}
   \ord_V \xi_k + \eta (k) \neq \ord_V \xi_{k'} + \eta (k')
\end{displaymath}
for two different weights involved of $T$  is dense. We 
find an $\eta$ in this set whose matrix has positive entries.   
Moreover we may assume that $\eta(k) < (q_1 - q)/2$ if
$\gamma^{(\delta)}[n](\xi_k X^k) \neq \infty$. Then $\delta (l) = 
\delta + l^{-1}\eta$ meets our requirements. \end{proof}

\smallskip
 
{\bf Remark:} In the case of a polynomial algebra $A$ it is useful to consider a stronger version
of overconvergence, which makes only sense for rings of Witt
vectors. With the notations above we define: 
\begin{equation}\label{A14e}
\breve{\gamma}_{\varepsilon}(\alpha) = \inf \{\ord_V \xi_k - \varepsilon |k| -u \}.
\end{equation}
This is clearly a pseudovaluation for each $\varepsilon$. If this $\inf$
is not $-\infty$ we call $\alpha$ overconvergent with respect to
$\breve{\gamma}_{\varepsilon}$. One easily verifies:
\begin{equation}\label{A16e}
\begin{array}{rcl}
\breve{\gamma}_{\varepsilon}(\alpha) & \leq & \gamma_{\varepsilon} (\alpha)\\
\breve{\gamma}_{\varepsilon}(\alpha^r) & \geq & (r-1)\gamma_{\varepsilon}(\alpha) 
+ \breve{\gamma}_{\varepsilon}(\alpha).
\end{array}
\end{equation}
It is important to note that the Teichm\"uller representative $[f]$ of an element $f \in A$
is $\breve{\gamma}_{\varepsilon}$-overconvergent. This is an immediate
consequence of the following  
\begin{lemma}\label{A15l}
Let $R$ be a $\mathbb{Z}_p$-algebra. Let $A$ be an $R$-algebra. Let
$x_1, \ldots, x_n 
\in R$ and $t_1, \ldots, t_d \in A$ be elements. We denote by $k =
(k_1, \ldots, k_d) \in \mathbb{Z}_{\geq 0}[1/p]$ a weight. Then we
have in $W(A)$ the following relation:

\begin{equation}\label{GN11e0}
   [x_1t_1 + \ldots + x_dt_d] = \sum_{k, |k| = 1} \alpha_{k}
   [t_1]^{k_1}\cdot \ldots \cdot [t_d]^{k_d},
\end{equation}
where $\alpha_k \in V^uW(R)$ and $p^u$ is the denominator of $k$.
\end{lemma}
\begin{proof} Clearly it is enough to show this Lemma in the case, where
$x_1 = 1, \ldots, x_d = 1$. Moreover we may restrict to the case where
$R = \mathbb{Z}_p$ and $A$ is the polynomial algebra over
$\mathbb{Z}_p$ in the variables $t_1, \ldots, t_d$. Then $W(A)$ is a
$\mathbb{Z}_{\geq 0}[1/p]$-graded, such that the monomial
$[t_1]^{k_1}\cdot \ldots \cdot [t_d]^{k_d}$ has degree $|k|$ (note that
this monomial is in general not in $W(A)$.) More precisely a Witt
vector of polynomials $(p_0, p_1, p_2, \ldots ) \neq 0$ is homogeneous of
degree $m \in \mathbb{Z}_{\geq 0}[1/p]$ if each polynomial $p_i$ has
degree $p^i m$, if $p_i \neq 0$. In the case where $p^im $ is not an
integer the condition says that $p_i = 0$. 

Since $[t_1 + \ldots + t_d]$ is homogeneous of degree $1$ the Lemma
follows from \cite{LZ} Prop.2.3. 
\end{proof}

\smallskip

\begin{lemma}\label{A18l}
Let $(A,\nu)$ a ring with a proper pseudovaluation. Let $\alpha \in VW(A)$.
Assume that $\gamma_{\varepsilon} (\alpha) \geq 0$. Then the element $1 - \alpha$ is a unit in
$W(A)$ and we have
\begin{equation}\label{A18e}
\gamma_{\varepsilon}(1- \alpha)^{-1} \geq 0.
\end{equation}

Assume moreover that $A = R[T_1, \ldots, T_d]$ is a polynomial ring with a degree valuation.. Then
\begin{displaymath}
\breve{\gamma}_{\varepsilon}(1- \alpha)^{-1} \geq \min \{0, \breve{\gamma}_{\varepsilon}(\alpha)\}.
\end{displaymath}
In particular $(1- \alpha)^{-1} $ is
$\breve{\gamma}_{\varepsilon}$-overconvergent if $\alpha$ is. 
\end{lemma}
 
\begin{proof} We write $\alpha = ~^V \eta.$ We find $\gamma_{\varepsilon/p} (\eta) > -1$.
We have in $W(A)$ the identity:
\begin{displaymath}
(1 - ~^V\eta)^{-1} = 1 + \sum_{i > 0} p^{i-1}~^V(\eta^{i}) = \sum_{i \geq 0} \alpha^i.
\end{displaymath}
The middle term shows that the series converges $V$-adically and the last sum proves the 
inequality (\ref{A18e}).
The last assertion is obvious from (\ref{A16e}).
\end{proof}

\begin{prop}\label{A18p}
Let $(A,\nu)$ a ring with a proper pseudovaluation. Let $\mathbf{w}_n: W(A) \rightarrow A$ denote 
the Witt polynomials. An element $\alpha \in W^{\dagger}(A)$ is a unit, iff $\mathbf{w}_0(\alpha)$ 
is a unit in $A$.

Assume moreover that $A = R[T_1, \ldots, T_d]$ is a polynomial ring
with a degree valuation. If $\alpha$ is
$\breve{\gamma}_{\varepsilon}$-overconvergent, then $\alpha^{-1}$ is
$\breve{\gamma}_{\delta}$-overconvergent for some $\delta >0$.
\end{prop}

\begin{proof} We write $\alpha = [a] + ~^V\eta$, with $a \in A$ and
$\eta \in W(A)$. To prove the first assertion we may assume that $a 
=1$. Applying Corollary \ref{A10c} we assume that $\gamma_{\varepsilon}
(~V\eta) >0$. Then the assertion follows from  Lemma \ref{A18l}.

Now we prove the second assertion:  Since
every Teichm\"uller representative is
$\breve{\gamma}_{\varepsilon}$-overconvergent, it suffices to show that
the inverse of $1 + [a^{-1}] ~^V\eta = 1 + ~^V ([a^{-p}]\eta)$ is
$\breve{\gamma}_{\varepsilon}$-overconvergent. Since
$\breve{\gamma}_{\varepsilon}$ is a pseudovaluation we see that $~^V
([a^{-p}]\eta)$ is $\breve{\gamma}_{\varepsilon}$-overconvergent too. By
Corollary \ref{A10c} we find $\varepsilon/p$ such that
\begin{displaymath}
  \gamma_{\varepsilon/p} ([a^{-p}]\eta) > -1.
\end{displaymath}
Therefore we may apply the Lemma \ref{A18l}. 
\end{proof}

\begin{prop}
Let $A$ be an algebra over a perfect field $K$. Let $\nu$ be an
admissible pseudovaluation on $A$. Then $W^{\dagger}(A)$ is an algebra
over the complete local ring $W(K)$.  

The $W(K)$-algebra $W^{\dagger}(A)$ is weakly complete in the sense
of \cite{MW}.  
\end{prop}

\begin{proof} 
Let $z_1, \ldots, z_r \in W^{\dagger}(A)$. Consider an infinite series 
\begin{equation}\label{A20e}
\sum a_k z^k, \quad a_k \in W(K), \quad z^k = z_1^{k_1}\cdot \ldots \cdot z_r^{k_r}.
\end{equation}
We assume that there are real numbers $\delta > 0$, and  $c$, such that
\begin{displaymath}
\ord_p a_k \geq \delta |k| + c.
\end{displaymath}
This implies that the series (\ref{A20e}) converges in $W(A)$. We have to show
that the series converges to an element $W^{\dagger}(A)$.
We choose a common radius $\varepsilon$ of convergence for $z_1, \ldots, z_r$. Making $\varepsilon$ smaller
we may assume that:
\begin{displaymath}
\gamma_{\varepsilon}(z_i) \geq -\delta.
\end{displaymath} 
Then we find:
\begin{displaymath}
\gamma_{\varepsilon}(a_k a^k) \geq \ord_p a_k - \delta |k| \geq c.
\end{displaymath}
Therefore (\ref{A20e}) converges to an element of $W^{\dagger}(A)$.
\end{proof}

We will point out that by Monsky and Washnitzer the last proposition
implies Hensel's Lemma for the overconvergent Witt vectors:

\begin{prop}\label{Hensel}
Let $A$ be an algebra over a perfect field $K$. Let $\nu$ be an
admissible pseudovaluation on $A$. Let $f(T) \in W^{\dagger}(A)[T]$
be a polynomial. We consider the homomorphism $\mathbf{w}_0 :
W^{\dagger}(A) \rightarrow  A.$

Let $a \in A$ be an element, such that 
\begin{displaymath}
   f(a) = 0 \quad \text{and} \; f'(a) \; \text{is a unit in} \; A.
\end{displaymath} 
Then there is a unique $\alpha \in W^{\dagger}(A)$ such that
$f(\alpha) = 0$ and such that $a \equiv \alpha
\;\text{mod}\;VW^{\dagger}(A)$. 
\end{prop}
\begin{proof} The kernel of the natural morphism
$W^{\dagger}(A)/pW^{\dagger}(A) \rightarrow A$ is an ideal whose
square is zero. Therefore there is an $\bar{\alpha} \in
W^{\dagger}(A)/pW^{\dagger}(A)$ which reduces to $a$ and such that
$f(\bar{\alpha}) = 0$. The rest of the proof is a general fact about 
weakly complete algebras explained below. 
\end{proof}

For the explanation we follow the notations of \cite{MW}: 
Let $(R,I)$ be a complete noetherian ring. 
Let $A$ be a weakly complete finitely generated (w.c.f.g.) algebra
over $(R,I)$. We write $\bar{A} = A/IA$.
Let $A\to B$ be a morphism of w.c.f.g. algebras, such that $\bar
B=\bar A[X_1,\dots, X_n]/(\bar F^{(1)}\dots \bar F^{(s)})$, $s\le n$
and the $s\times s$ subdeterminants of $(\frac {\partial
  F^{(i)}}{\partial X_j})$ generate the unit ideal in $\bar B$.
Then by \cite{MW} p. 195 the morphism $A\to B$ is very smooth.
As an example we may take for $B$ the weak completion of
   $$A[X,T]/(f(X), 1-f'(X)T),$$
where $f(X)\in A[X]$ is a polynomial.

\begin{prop} Let $C$ be a weakly complete (not necessarily finitely
  generated but $p$-adically separated)   algebra over $(R,I)$. Let
  $f(X)\in C[X]$ be a polynomial and let 
  $\bar\gamma\in\bar C$ be an element, such that $f(\bar\gamma)=0$ and
  $f'(\bar\gamma)$ is a unit in $\bar C$. Then there is a unique
  element $\gamma\in C$, such that $f(\gamma)=0$ and
  $\gamma\equiv\bar\gamma \; \mod \, IC$.
\end{prop}

\begin{proof} By Hensel's Lemma applied to the completion of $C$ the
uniqueness of the solution is clear.

For the existence we write
$f(X)=s_dX^d+s_{d-1}X^{d-1}+\dots+s_1X+s_0$,
where $s_i\in C$.

Let $A=R[S_d,\dots, S_0]^{\dagger}$ be the weak completion of the
polynomial algebra. We set 
    $$F(X)=S_dX^d+\dots+S_1X+S_0\in A[X]$$ 
and we let $B$ be the weak completion of
   $$A[X,T]/(F(X), 1-TF'(X)).$$

Let $A\to C$ be the homomorphism defined by $S_i\mapsto s_i$. The
solution $\bar\gamma$ defines a homomorphism
   $$R/I[S_d,\dots, S_0, X, T]/(\bar F(X), 1-T\bar F'(X))\to \bar C$$
where $S_i\mapsto s_i\mod IC$ and $X\mapsto\bar\gamma, T\mapsto
f'(\bar\gamma)^{-1}$.

Hence we obtain a commutative diagramm
     \eqn\label{1}\xymatrix{ A\ar[d] \ar[r]&C\ar[d]\\
      B \ar[r]&\bar C.}\eeqn

Since $A\to B$ is very smooth by the example above, we find a morphism
$B'\to C$ making $(\ref{1})$ commutative. The image of $X$ is the
desired solution $\gamma \in C$. 
\end{proof}

We will now study the behaviour of overconvergent Witt vectors in 
finite \'etale extensions. Let $A$ be a finitely generated
$K$-algebra. Let $B$ a finite \'etale $A$-algebra which is  free as an
$A$-module. Let $e_i$, $1 \leq i \leq r$ be a basis of the $A$-module
$B$.  Then the natural map 
\begin{equation}\label{A22e}
W(A)^r \rightarrow W(B),
\end{equation}
which maps the standard basis of the free module $W(A)^r$ to the
Teichm\"uller representatives 
$[e_i]$ is an isomorphism. Moreover $W(B)$ is an \'etale algebra over $W(A)$. 

Indeed, by \cite{LZ} A8 the $W_n(A)$-algebra $W_n(B)$ is \'etale for
each $n$. We set 
$I_n = VW_{n-1}(A) \subset W_n(A)$. Then by loc.cit. we have
$I_nW_n(B) \subset VW_n(B)$. 
>From this we conclude by the lemma of Nakayama that:
\begin{displaymath}
W_n(A)^r \rightarrow W_n(B),
\end{displaymath}
is an isomorphism. Taking the projective limit we obtain
(\ref{A22e}). If we tensor (\ref{A22e}) 
with $A \otimes_{\mathbf{w}_0}$ we obtain that $A
\otimes_{\mathbf{w}_0} W(B) = B$.

\smallskip
 
We will now assume that $B$ is monic
\begin{displaymath}
    B = A[T]/f(T)A[T], 
\end{displaymath}
where
\begin{equation}
   f(T) = T^m - c_{m-1}T^{m-1} - \ldots - c_1T -c_0.
\end{equation}
Let $\nu$ be a negative pseudovaluation on $A$. We endow $B$ with the
equivalence class of admissible pseudovaluations defined by Proposition
\ref{A1p}. 
\begin{lemma}\label{A20l}
Let $d \in \mathbb{R}$, such that $d > \nu(c_i)$ for $i = 1, \ldots,
m$. An element $b \in B$ has a unique representation
\begin{displaymath}
   b = \sum_{i=0}^{m-1} a_iT^i.
\end{displaymath}
We set 
\begin{equation}\label{AEX1e}
   \tilde{\nu}(b) = \min_{i=1,\ldots, m-1} \{\nu(a_i) - id \}.
\end{equation}
Then $\tilde{\nu}$ is an admissible pseudovaluation on $B$.
\end{lemma}

\begin{proof} 
We consider on $A[T]$ the pseudovaluation $\mu$
(\ref{A2e}). We will show that with $d$ as above $\tilde{\nu}$ 
is the quotient of $\mu$.

Let 
\begin{displaymath}
   \tilde{b} = \sum_{j=0}^{s} u_j T^j,
\end{displaymath}
be an arbitrary representative of $b$. We need to show that
$\mu(\tilde{b})$ is smaller than the right hand side of (\ref{AEX1e}).
We prove this by induction on $s$. For $s <m$ there is nothing to
show. For $s \geq m$ we obtain another representative of $b$:
\begin{equation}
   \tilde{b}' = \sum_{j=0}^{m-1}u_j T^j + \sum_{k \geq m}
   u_k(\sum_{l=0}^{m-1} c_l T^l)T^{k-m}.
\end{equation}  
On the right hand side there is a polynomial of degree at most $s-1$.
Therefore it suffices by induction to show that 
\begin{displaymath}
   \mu(\tilde{b}') \geq \mu(\tilde{b}).
\end{displaymath}
The last inequality is a consequence of the following:
\begin{equation}
\begin{array}{lclr}
\mu(u_j T^j) & \geq & \mu(\tilde{b}), & \text{for} \; j = 0, \ldots,
m-1\\
\mu(u_k c_l T^{k-m+l}) & \geq & \mu(\tilde{b}) & \text{for} \; k \geq m, \;
0 \leq l \leq m-1. 
\end{array}
\end{equation}
The first set of these inequalities is trivial. For the second set we
compute
\begin{displaymath}
\begin{array}{lcl}
  \mu(u_k c_l T^{k-m+l}) & \geq & \nu(u_k) + \nu(c_l) - kd +(m-l)d\\
& \geq & \mu(u_k) - kd \geq \mu(\tilde{b}).
\end{array}
\end{displaymath} 
The last equation holds because by the choice of $d$:
\begin{displaymath}
   \nu(c_l) + (m-l)d \geq 0.
\end{displaymath}
This shows the second set of inequalities. 
\end{proof}

Because $\tilde{\nu}$ restricted to $A$ coincides with $\nu$ we
simplify the notation by setting $\tilde{\nu} = \nu$. The Gauss norms 
(\ref{A8e}) induced by the pseudovaluation $\nu$ on $W(B)$ and $W(A)$
will be also denoted by the same symbols $\gamma_{\varepsilon}$.

\begin{lemma}\label{A22l}
With the notations of Lemma \ref{A20l} we assume that $B$ is \`etale
over $A$. We will denote the residue class of $T$ in $B$ by $t$.

Then there is a constant $G \in \mathbb{R}$ with the
following property: Each $b\in B$ has for each integer $n \geq 0$ a
unique representation 
\begin{displaymath}
   b = \sum_{i=0}^{m-1} a_{ni}t^{ip^n}.
\end{displaymath}
Then we have the following estimates for the pseudovaluations of
$a_{ni}$: 
\begin{equation}
   \nu(a_{ni}) \geq \nu(b) - p^n G.
\end{equation}
\end{lemma}

\begin{proof} Since $B$ is \'etale over $A$ the elements
\begin{displaymath}
   1, t^{p^n}, t^{2p^n}, \ldots, t^{(m-1)p^n}
\end{displaymath}
are for each $n$ a basis of the $A$-module $B$. 
We write
\begin{equation}
   t^i = \sum_{j=0}^{m-1} u_{ji}t^{jp}.
\end{equation}
We introduce the matrix $U = (u_{ji})$ and we set for matrices
\begin{displaymath}
   \nu(U) = \min_{i,j} \{ \nu (u_{ji}) \}.
\end{displaymath}
We deduce the relation:
\begin{displaymath}
   a_{1j} = \sum_{i} u_{ji}a_{0i}.
\end{displaymath}
We will write the last equality in matrix notation:
\begin{displaymath}
   a(1) = Ua(0).
\end{displaymath}
Let $U^{(p^{n})}$ the matrix obtained form $U$ by raising all
entries of $U$ in the $p^{n}$-th power. Then we obtain with the
obvious notation:
\begin{displaymath}
   a(n+1) = U^{(p^{n})}a(n).
\end{displaymath} 
It is obvious that for two matrices $U_1, U_2$ with entries in $B$
\begin{displaymath}
   \nu(U_1U_2) \geq \nu(U_1) + \nu(U_2).
\end{displaymath}
We choose a constant $C$ such that 
\begin{displaymath}
   \nu (U) \geq -C.
\end{displaymath}
Therefore we obtain:
\begin{displaymath}
\begin{array}{lcl}
  \nu(a(n)) & = & \nu(U^{(p^{n-1})}\cdot \ldots \cdot U a(0))\\ 
            & \geq & -(p^{n-1}C + \ldots + pC + C) + \nu(a(0)).
\end{array}
\end{displaymath}
By Lemma \ref{A22e} we have 
\begin{displaymath}
   \nu(b) = \min \{\nu(a_{0i}) -id  \} \leq \nu(a(0)).
\end{displaymath}
Therefore we obtain:
\begin{displaymath}
   \nu (a(n)) \geq - p^n\frac{C}{p-1} + \nu(b).
\end{displaymath}
We therefore found the desired constant. 
\end{proof}

\begin{prop}
Let $B = A[t]$ be a finite \'etale $A$-algebra as in Lemma \ref{A22l}.
Let $G > 0$ be the constant of this Lemma. 
Let $x = [t] \in W(B)$ be the Teichm\"uller representative. By
(\ref{A22e}) $1, x, \ldots, x^{m-1}$ is a basis of the $W(A)$-module
$W(B)$. We write an element $\eta \in W(B)$ 
\begin{displaymath}
   \eta = \sum _{i=0}^{m-1} \xi_i x^i, \quad \xi_i \in W(A).
\end{displaymath}

There is a real number $\delta >0$, such that for $\varepsilon \leq
\delta$ an inequality
\begin{displaymath}
   \gamma_{\varepsilon}(\eta) \geq -C \quad \text{implies} \quad 
   \gamma_{\varepsilon}(\xi_i) \geq -C -\varepsilon G.
\end{displaymath}
\end{prop}

\begin{proof}
We choose a constant $G' >0$ such that 
\begin{displaymath}
   \nu (t^i) \geq -G', \quad \text{for} \; i = 0, \ldots, m-1.
\end{displaymath} 
We choose $\delta$ such that $\delta(G + G')  \leq 1$. We write
\begin{displaymath}
   \xi_i = \sum_{s \geq 0} ~^{V^s}[a_{s,i}] \quad \text{with} \;
   a_{s,i} \in A.
\end{displaymath}
We define
\begin{displaymath}
   \zeta_i(n) = \sum_{s \geq n} ~^{V^{s-n}} [a_{s,i}].
\end{displaymath}
We will show by induction on $n$ the following two assertions:
\begin{equation}\label{A24e}
   \gamma_{\varepsilon}(\sum_{i=0}^{m-1} ~^{V^{n}}\zeta_i(n) x^i) \geq -C. 
\end{equation}

\begin{equation}\label{A26e}
   \gamma_{\varepsilon} (~^{V^{n}}[a_{n,i}]) \geq -C - \varepsilon G.
\end{equation}
We begin to show that the first inequality for a given $n$ implies
the second. We set:
\begin{displaymath}
   \theta(n) = \sum_{i=0}^{m-1} ~^{V^{n}}\zeta_i(n) x^i.
\end{displaymath}
The first non-zero component of this Witt vector is 
\begin{displaymath}
   y_n = \sum_{i=0}^{m-1} a_{n,i}t^{ip^n}.
\end{displaymath}
in place $n+1$. We conclude that
\begin{displaymath}
   n + \varepsilon p^{-n} \nu (y_n) \geq \gamma_{\epsilon}(\theta(n))
   \geq -C, 
\end{displaymath}
where the last inequality is (\ref{A24e}). This shows that:
\begin{displaymath}
   \nu (y) \geq - \varepsilon p^{n} (C+n).
\end{displaymath}
We conclude by Lemma \ref{A22l} that
\begin{equation}\label{A28e}
   \nu (a_{n,i}) \geq -  (p^{n}/\varepsilon) (C+n) - p^n G,
\end{equation}
and therefore 
\begin{displaymath}
   \gamma_{\varepsilon} (~^{V^{n}}[a_{n,i}]) = n + \varepsilon p^{-n} \nu
   (a_{n,i}) \geq - C - \varepsilon G.
\end{displaymath}
Therefore the Proposition follows if we show the assertion (\ref{A24e}) 
by induction. The assertion is trivial for $n=0$ and we assume it for
$n$. With the notation above we write:
\begin{displaymath}
\begin{array}{lcl}
   \theta(n+1) & = & \theta(n) - \sum_{i=0}^{m-1}
   ~^{V^{n}}[a_{n,i}]x^i\\
& = & (\theta(n) - ~^{V^{n}}[y_n]) - (\sum_{i=0}^{m-1}
~^{V^{n}}[a_{n,i}]x^i  - ~^{V^{n}}[y_n]). 
\end{array}
\end{displaymath}
The Witt vector in the first brackets has only entries which also 
appear in $\theta(n)$ and therefore has Gauss norm $\gamma_{\varepsilon}
\geq -C$. The assertion follows if we show the same inequality for the Witt
vector in the second brackets:
\begin{displaymath}
   ~^{V^{n+1}}\tau = (\sum_{i=0}^{m-1}~^{V^{n}}[a_{n,i}]x^i  - 
~^{V^{n}}[y_n]).
\end{displaymath}
We set 
\begin{displaymath}
  [y_n] = \sum_{i=0}^{m-1} [a_{n,i}t^{ip^n}] = (s_0,s_1,s_2, \ldots ).
\end{displaymath}
Then we find
\begin{displaymath}
   ~^V\tau = (0,s_1,s_2, \ldots ).
\end{displaymath}
We know that $s_l$ is a homogeneous polynomial of degree $p^l$ in 
the variables $a_{n,i}t^{ip^n}$ for $i= 1, \ldots m-1$. By the 
choice of $G'$ we find $\nu (t^{ip^n}) \geq p^n\nu(t^i) \geq - p^n
G'$. Using (\ref{A28e}) we find:
\begin{displaymath}
   \nu (s_l) \geq -p^l((p^n/\varepsilon)(C+n) + p^nG) - p^lp^nG' =
-p^{n+l}((1/\varepsilon)(C+n) + G + G').
\end{displaymath}  
We have 
\begin{displaymath}
   ~^{V^{n+1}}\tau = \sum_{l \geq 1} ~^{V^{n+l}}[s_l]. 
\end{displaymath}
For the Gauss norms of the entries of this vector we find for $l \geq 1$:
\begin{displaymath}
\begin{array}{lcl}
\gamma_{\varepsilon}( ~^{V^{n+l}}[s_l]) & = & n + l + \varepsilon
   p^{-n-l}\nu(s_l) \geq n + l - \varepsilon ((1/\varepsilon)(C+n) + G
   +G')\\
& = & l - C - \varepsilon (G + G') \geq -C.
\end{array}
\end{displaymath}
The last inequality follows since $l\geq 1$ by the choice of $\delta$.
We conclude that 
\begin{displaymath}
\gamma_{\varepsilon}(~^{V^{n+1}}\tau) \geq -C.
\end{displaymath}
\end{proof}

\begin{cor}\label{endl-et}
Let $A$ be a finitely generated algebra over  $K$. Let $B = A[T]/(f(T))$ be a finite
\'etale $A$-algebra, where $f(T) \in A[T]$ is a monic polynomial of
degree $n$. We
denote by $t$ the residue class of $T$ in $B$. We set
$x = [t] \in W^{\dagger} (B)$. 

Then $W^{\dagger} (B)$ is finite and \'etale over $W^{\dagger} (A)$
with basis $1,x \ldots, x^{n-1}$. 
\end{cor}




\newpage
\begin{thebibliography}{XXX}
\bibitem{DLZ} Ch.Davis, A.Langer, Th.Zink, Overconvergent de Rham-Witt
  Cohomology.  Preprint  

\bibitem{J} de Jong, A. J.  "Homomorphisms of Barsotti-Tate groups and
crystals in positive characteristic".  Invent. Math.  134  (1998),
no. 2,  301 -333.  

\bibitem{K1} K. Kedlaya, More \'{E}tale Curves of Affine Spaces in Positive
Characteristic, J. Algebraic Geometry \textbf{14} (2005), 187--192.

\bibitem{K2} Kedlaya, Kiran S. Slope filtrations revisited.
  Doc. Math.10 (2005).

\bibitem{LZ} A. Langer, Th. Zink, De Rham-Witt Cohomology for a
Proper and Smooth Morphism, Journal Inst. Math. Jussieu
\textbf{3(2)} (2004), 231--314.

\bibitem{LZ2} A. Langer, Th. Zink, Gauss-Manin Connection via Witt
Differentials, Nagoya Math. J. \textbf{179} (2005), 1--16.

\bibitem{L} Lubkin, Saul Generalization of $p$-adic cohomology: bounded Witt
vectors. A canonical lifting of a variety in characteristic $p\not=0$
back to characteristic zero.  Compositio Math.  34  (1977), no. 3,
225--277. 

\bibitem{Me} D. Meredith, Weak Formal Schemes, Nagoya Math. J.
\textbf{45} (1971), 1--38.

\bibitem{MW} P. Monsky, G. Washnitzer, Formal Cohomology I, Ann.
Math. \textbf{88(2)} (1968), 181--217.

\bibitem{M} P. Monsky, Formal Cohomology II, Ann. Math. \textbf{88}
(1968), 218--238.
\end{thebibliography}
\end{document}